\documentclass[10pt,a4paper]{amsart}

\usepackage{lmodern}
\usepackage[utf8]{inputenc}
\usepackage[T1]{fontenc}

\usepackage{fancyref}
\usepackage[colorlinks=true]{hyperref}

\usepackage{amsmath}
\usepackage{amssymb}
\usepackage{amsthm}

\usepackage{tikz-cd}

\newtheorem{theo}{Theorem}[section]
\newtheorem{prop}[theo]{Proposition}
\newtheorem{lemm}[theo]{Lemma}
\newtheorem{coro}[theo]{Corollary}

\theoremstyle{definition}
\newtheorem{defi}[theo]{Definition}
\newtheorem{exam}[theo]{Example}

\newcommand{\kk}[1]{\mathbb{#1}}
\newcommand{\cc}[1]{\mathcal{#1}}

\def\^#1{^{[#1]}}

\def\ov#1{\overline{#1}}

\DeclareMathOperator{\Span}{span}
\DeclareMathOperator{\Ric}{Ric}
\DeclareMathOperator{\Gr}{Gr}

\DeclareMathOperator{\im}{Im}

\DeclareMathOperator{\Ker}{Ker}
\DeclareMathOperator{\End}{End}

\DeclareMathOperator{\Hom}{Hom}
\DeclareMathOperator{\id}{id}
\DeclareMathOperator{\tr}{tr}

\def\snd{\sigma}

\author{Gunnar Þór Magnússon}
\address{Hafnarfjörður, Iceland}
\email{gunnar@magnusson.io}
\date{\today}
\title[Degenerate Hermitian geometry]{Degenerate Hermitian geometry and\\
curvatures of holomorphic fibrations}

\hypersetup{
 pdfauthor={Gunnar Þór Magnússon},
 pdftitle={Degenerate Hermitian geometry},
 pdfkeywords={},
 pdfsubject={},
 pdfcreator={Emacs 27.1 (Org mode 9.3)},
 pdflang={English}}

\begin{document}

\begin{abstract}
We calculate curvature tensors of metrics on the total spaces of holomorphic fibrations. Our main tool is a theory of Chern connections and curvature forms for possibly degenerate Hermitian forms on holomorphic vector bundles. We prove a version of the Codazzi--Griffiths equations for curvatures of sub- and quotient bundles in that setting and apply them to the study of honest Hermitian metrics on fibrations. We apply this to calculate the curvature of a metric on Grassmannian bundles and prove they have positive holomorphic sectional curvature if the base does.
\end{abstract}

\maketitle

\section*{Introduction}

Let $\pi : X \to S$ be a family of complex manifolds, and consider the associated short exact sequence
\[
\begin{tikzcd}
0 \ar[r] & T_{X/S} \ar[r] & T_X \ar[r] & \pi^*T_S \ar[r] & 0.
\end{tikzcd}
\]
We often find ourselves in the situation of having a metric $h_S$ on the base $S$ and a Hermitian form $h_{X/S}$ on the $X$ that is positive-definite on the relative tangent bundle $T_{X/S}$, and then construct a metric on $X$ as the sum $h_X := h_{X/S} + \pi^*h_S$. This happens when we consider a projective or Grassmannian bundle associated to a vector bundle~\cite{alvarez2016positive,alvarez2018projectivized,yang2019hirzebruch}; the blow-up of a point like in the original proof of the Kodaira embedding theorem~\cite{kodaira-embedding}; metrics on holomorphic fibrations~\cite{calabi-fibres-holomorphes}; or metrics on versal families of deformations~\cite{magnusson2012natural}.

In these situations, the metric on the base has a curvature tensor, and the relative form often has something that resembles such a tensor. As neither form is positive-definite on $X$, we can't use our usual tools to calculate a curvature tensor for $h_X$ from this information. Instead, we are often reduced to committing difficult calculations in local coordinates; see for example Chaturvedi and Heier~\cite[Theorem~1.1]{chaturvedi2020hermitian}.

If we had the equivalent of Chern connections and curvature forms for arbitrary Hermitian forms, we could try to prove a version of the Codazzi--Griffiths equations~\cite{griffiths1965hermitian} that relate the curvature form of sub- and quotient bundles to that of the ambient bundle. In the positive-definite case, those equations can be used to calculate the curvature tensor of the sum $h_1 + h_2$ of two metrics on a given bundle $E$ by considering the short exact sequence
\[
\begin{tikzcd}
0 \ar[r] & E \ar[r] & E\oplus E \ar[r] &  E \ar[r] & 0,
\end{tikzcd}
\]
where the metrics on the bundles are (in order) $h_1 + h_2$, $h_1 \oplus h_2$, and the induced metric on the ``quotient''; see \cite[Chapter~7, Exercise~6]{zheng2000complex}.

In this note, we fill in the details of this sketch. In section~\ref{sec:degenerate-chern-connections} we show that a holomorphic vector bundle $E \to X$ with a smooth Hermitian form $b$ admits a Chern connection. Such a connection is not unique when the form is degenerate, but we can always get rid of the ambiguity between different choices of connections by applying the form $b$, as the difference of two such connections must take values in its kernel. This yields a unique curvature \emph{tensor} for such a form. We then prove the Codazzi--Griffiths equations hold in this setting and calculate what they imply for sums of Hermitian forms.

As an application, in section~\ref{sec:grassmannian-bundles} we first recover a result of Chaturvedi and Heier~\cite{chaturvedi2020hermitian}, and then dot some i's and cross some t's from works of \'Alvarez~\cite{alvarez2016positive} and \'Alvarez, Heier and Zheng~\cite{alvarez2018projectivized} on the curvature of metrics on projectivized vector bundles (see also Yang and Zheng~\cite{yang2019hirzebruch}). They prove that such bundles admit metrics of positive holomorphic sectional curvature, and finish by saying the same arguments can prove that so do Grassmannian bundles. We treat the slightly more general case of Grassmannian bundles without mention of an underlying vector bundle, and as a corollary conclude that bundles of flag manifolds admit metrics of positive holomorphic sectional curvature once the base manifold does.
Along with our insistence on coordinate-invariant calculations, these are the new results of our~note.%


To do this, we require a degenerate analogue of many statements about adjoints, orthogonal subspaces and induced inner products that are classical in the positive-definite case. We do not know a reference for these results, which would make for fun exercises in a linear algebra course, and thus provide statements and proofs. We do this in section~\ref{sec:degenerate-linear-algebra} as we need to refer to those results for our work on connections and curvature.

A seeming novelty of our approach is that we prove all our results using the connection formulation of complex differential geometry (versus the local coordinate or moving frame formulations). While this connection approach is common in modern Riemannian geometry, complex geometers seem to prefer the older, index-heavy, ways. We hope this note also serves as an advertisement for the merits of the connection approach.

There is some prior art on degenerate metrics in Riemannian geometry; see for example \cite{bel1975degenerate,stoica2011cartan,stoica2014singular}.
The very quick version seems to be that one can define a Levi-Civita connection for such an object and thus a curvature tensor, but neither will generally be unique.
This nicely mirrors what we find here.

We finish this introduction by noting that we think the technical results developed here on general Hermitian forms are only interesting when applied to the study of honest Hermitian metrics, and do not have merit on their own.
Some examples may explain this opinion:

\begin{exam}
More degenerate the form is, the less information the connection gives:
Let $E \to X$ be a holomorphic vector bundle, let $b = 0$, and let $D$ be any connection whose $(0,1)$-part is $\bar\partial$. Then $D$ is a Chern connection of $b$.
\end{exam}

\begin{exam}
Let $X$ be a compact complex manifold of dimension $n$
that admits a nowhere-zero holomorphic one-form $\alpha$.
It  gives us a smooth Hermitian form $b = i \alpha \wedge \ov \alpha$ whose kernel forms a holomorphic subbundle of rank $n-1$ and whose quotient is the trivial line bundle.
Any of our Chern connections for this form will yield a curvature form whose top exterior power vanishes; so if we were to develop Chern--Weil theory in this setting we should expect to conclude the top Chern class of the manifold is zero.
However, a much easier way to conclude the same is to use the additivity property of total Chern classes on the short exact sequence we began with.
\end{exam}

\begin{exam}
Even when a form is non-degenerate, having a Chern connection and curvature form just doesn't let us say very much:
Let $\pi : X \to S$ be a family of compact complex K\"ahler manifolds of dimension $2n$ over a smooth base $S$.
Let $E \to S$ be the holomorphic vector bundle whose sheaf of sections is $\mathcal R^{n}\pi_{*}\mathbb C \otimes \mathcal O_{S}$ and whose fiber over a point $s$ is $E_{s} = H^{n}(X_{s}, \mathbb C)$.
This vector bundle admits a non-degenerate Hermitian form given by the cup product, and the flat Gauss--Manin connection is the Chern connection of this Hermitian form.
\end{exam}

Also note that in the non-degenerate case the most useful results are found when the form is positive-definite. For example the Laplacian of a non-degenerate form of mixed signature will not be strongly elliptic,\footnote{For example, on the torus $\kk C^2 / \kk Z^4$ with the form $b = \smash{\frac i2} dz_1 \wedge d \bar z_1 - \smash{\frac i2} dz_2 \wedge d\bar z_2$ the kernel of the Laplacian is infinite-dimensional.} so the Hodge isomorphism theorem fails, and with it go most of the tools and techniques developed over the last hundred years or so.

We'll finish with one final example we don't have any applications in mind for, but find amusing in its inaccessibility.

\begin{exam}
Let $(E,h) \to X$ be a holomorphic vector bundle with a Hermitian metric $h$. The curvature tensor of the metric,
\[
  R(\alpha, \ov\beta, s, \ov t)
  = h(\tfrac i2 \Theta_{\alpha, \ov\beta}s, \ov t),
\]
defines a Hermitian form on $E \otimes T_{X}$ that is usually not positive-definite. Good luck saying anything about its curvature tensor.
\end{exam}

\section{Linear algebra}
\label{sec:degenerate-linear-algebra}

All vector spaces are finite-dimensional and over the complex numbers. Let $V$ be a vector space and let $b$ be a Hermitian form on $V$. We want to study the geometry of the pair $(V,b)$.\footnote{Everything we say also applies to symmetric bilinear forms on finite-dimensional vector spaces over arbitrary fields. The reader who is interested in those can mentally erase all the conjugations of vectors or spaces.}
Our objective is to establish a version of the Codazzi--Griffiths equations for sub- and quotient bundles of a holomorphic vector bundle, so the main questions we want to answer here are:
\begin{itemize}
\item
When does an adjoint of a linear morphism exist?

\item
Does a Hermitian form on $V$ induce Hermitian forms on all spaces derived from its tensor algebra?

\item
Given a short exact sequence $0 \to S \to V \to Q \to 0$ and a Hermitian form on $V$, do we get Hermitian forms on $S$ and $Q$?
\end{itemize}

\begin{defi}
A \emph{Hermitian space} is a finite-dimensional vector space $V$ with a Hermitian form $b$.
A \emph{morphism} of Hermitian spaces is a linear morphism $f : V \to W$ such that
\[
b_W(fx, \ov{fy}) = b_V(x, \ov y)
\]
for all $x, y \in V$.
\end{defi}

The category this forms is the one that has Hermitian spaces for objects and the above morphisms between them. Given two objects in this category, there is not necessarily any morphism between them; if $V$ has a non-degenerate form while $W$ has the zero form, there is no morphism $V \to W$.

A Hermitian form can be seen as a linear morphism $b : V \to \ov V^*$. As such, it has a kernel. We clearly have
\[
\Ker b = \{ x \in V \mid b(x, \ov y) = 0 \text{ for all $y \in V$}\}.
\]

The results we want to prove are all known for spaces with a non-degenerate form. Our main technical tool will be to reduce to that case by quotienting out the kernel of the form we have.

\begin{prop}
Let $V$ be a vector space. For every Hermitian form $b$ on $V$, there exists a unique non-degenerate Hermitian form $\hat b$ on $V / \Ker V$ such that $q : V \to V/\Ker b$ is a Hermitian morphism.
\end{prop}

\begin{proof}
We define
\[
\hat b(qx, \ov{qy})
= b(x, \ov y).
\]
This is well-defined, as any other element in $qx$ differs from the one we chose by an element of $\Ker q = \Ker b$. The form $\hat b$ is also non-degenerate, as if $qx \in \Ker \hat b$ then $x \in \Ker b$, so $qx = 0$. The form is also defined in such a way that $q$ is a morphism of Hermitian spaces.

If $b'$ is another non-degenerate Hermitian form on $V / \Ker b$ that makes $q$ into a Hermitian morphism, then
\[
b'(qx, \ov{qy})
= b(x, \ov{y})
= b(qx, \ov{qy})
\]
for all $qy \in V / \Ker b$. Thus $b(qx) = b'(qx)$ as $(0,1)$-forms for all $qx \in V / \Ker b$, so $b = b'$.
\end{proof}

If this operation should be called anything, it is clearly a \emph{purge} of the original space, as it removes all the degeneracy from the Hermitian form.

\subsection*{Existence of adjoints}
\label{sec:existence-adjoints}

\begin{defi}
Let $(V, b_V)$ and $(W, b_W)$ be Hermitian spaces, and let $f : V \to W$ be a linear morphism. An \emph{adjoint} of $f$ is a linear morphism $f^\dagger : W \to V$ such that
\[
b_V(f^\dagger x, \ov y)
= b_W(x, \ov{f y})
\]
for all $x \in W$ and $y \in V$.
\end{defi}

If $b_V$ is non-degenerate, every morphism $f : V \to W$ admits a unique adjoint. It is defined by
\[
f^\dagger = b_V^{-1} \circ \ov f^* \circ b_W,
\]
where $\ov f^*$ is the conjugate dual morphism $f^* : \ov W^* \to \ov V^*$ that $f$ induces. In general adjoints may not exist.

\begin{theo}
Let $f : V \to W$ be a linear morphism. The morphism $f$ admits an adjoint if and only if $f(\Ker b_V) \subset \Ker b_W$. If the set of adjoints of $f$ is not empty, it is a $\Hom(W, \Ker b_V)$ torsor.
\end{theo}

\begin{proof}
If $f$ admits an adjoint $f^\dagger$, then
\[
b_V(f^\dagger x, \ov y)
= b_W(x, \ov{fy})
\]
for all $x \in W$ and $y \in V$. Taking $y \in \Ker b$, we see that $b_W(x, \ov{fy}) = 0$ for all $x \in W$, so $f(y) \in \Ker b_W$.

Suppose then that $f(\Ker b_V) \subset \Ker b_W$. Then $f$ defines a linear morphism $\hat f : V / \Ker b_V \to W / \Ker b_W$. As the induced forms on those spaces are non-degenerate, we get a unique adjoint $\hat f^\dagger : W / \Ker b_W \to V / \Ker b_V$. It induces a morphism $W \to V / \Ker b_V$. Picking a lift $V / \Ker b_V \to V$ we then get an adjoint $f^\dagger$ of $f$.

If $g$ and $h$ are adjoints of $f^\dagger$, then
\[
b_V((g - h)x, \ov y)
= b_V(gx, \ov y) - b_V(hx, \ov y)
= b_W(x, \ov{fy}) - b_W(x, \ov{fy})
= 0
\]
for all $x \in W$ and $y \in V$, so the image of $g - h$ is contained in $\Ker b_V$.
\end{proof}

\begin{coro}
The subset of $\Hom(V,W)$ of morphisms that have adjoints is a linear subspace
of codimension $\dim \Ker b_V (\dim W - \dim \Ker b_W)$.
\end{coro}

\begin{proof}
The subset is not empty because the zero morphism always has an adjoint,
and it is clearly a subspace.

Let's pick splittings $V = \Ker b_V \oplus V / \Ker b_V$ and $W = \Ker b_W
\oplus W / \Ker b_W$. Then we get a splitting
$$
\displaylines{
\Hom(V, W)
= \Hom(\Ker b_V, \Ker b_W)
\oplus \Hom(\Ker b_V, W / \Ker b_W)
\hfill\cr\hfill{}
\oplus \Hom(V / \Ker b_v, \Ker b_W)
\oplus \Hom(V / \Ker b_v, W / \Ker b_W).
}
$$
A morphism $f : V \to W$ satisfies $f(\Ker b_V) \subset f(\Ker b_W)$ if and only
if its projection onto $\Hom(\Ker b_V, W / \Ker b_W)$ in this splitting is zero.
\end{proof}

\begin{coro}
If $f : V \to W$ has an adjoint $f^\dagger$, then $f^\dagger$ also has an adjoint.
\end{coro}

\begin{proof}
If $x \in \Ker W$, then
\[
b_V(f^\dagger x, \ov y)
= b_W(x, \ov{fy})
= 0
\]
for all $y \in V$, so $f^\dagger x \in \Ker b_V$.
\end{proof}

The difference between a morphism $f$ and an adjoint of its adjoint is an element of $\Hom(V, \Ker b_V)$, so a double adjoint is unique and equal to $f$ if and only if $b_V$ is non-degenerate.

\begin{prop}
Let $j : S \to V$ be a linear morphism, let $b_V$ be a Hermitian form on $V$, and set $b_S = j^*b_V$. Then $j$ admits an adjoint.
\end{prop}

\begin{proof}
We have
\[
\Ker b_S = j^{-1}(\Ker b_V) + \Ker j
\]
so $j(\Ker b_S) = j(j^{-1}(\Ker b_V)) \subset \Ker b_V$.
\end{proof}

\subsection*{Orthogonal complements}

The objective of this section is to show how to obtain a Hermitian form on a quotient space given one on the ambient space. It is well-known how to do this when the forms are non-degenerate (for example, take duals to embed the quotient dual in the ambient dual, restrict the dual metric, and invert) but less clear how this should work in the general case.

\begin{defi}
Let $(V,b)$ be a Hermitian space, and let $S \subset V$ be a subspace. The \emph{orthogonal complement} of $S$ is
\[
S^\perp = \{ v \in V \mid b(v, \ov w) = 0 \text{ for all $w \in S$}\}.
\]
\end{defi}

The orthogonal complement of a subspace is clearly again a subspace.

\begin{prop}
\[
\Ker b \subset S^\perp.
\]
\end{prop}

\begin{proof}
Let $v \in \Ker b$. Then $b(v, \ov w) = 0$ for all $w \in S$.
\end{proof}

As before we have the quotient space $q : V \to V / \Ker b$, and the induced non-degenerate form $\hat b$ on $V / \Ker b$.

\begin{prop}
Let $S \subset V$. Then
\[
q(S^\perp) = q(S)^\perp,
\]
where the orthogonal complement on the right-hand side is with respect to $\hat b$.
\end{prop}

\begin{proof}
Let $v \in S^\perp$, so $b(v, \ov w) = 0$ for all $w \in S$. Then $\hat b(qv, \ov{qw}) = 0$ for all $w \in S$, which means that $\hat b(qv, \ov w) = 0$ for all $w \in q(S)$, so $qv \in q(S)^\perp$.

Suppose then that $v \in q(S)^\perp$, so $\hat b(v, \ov w) = 0$ for all $w \in q(S)$. If $x \in V$ is such that $qx = v$, then this means that $b(x, \ov w) = 0$ for all $w \in S$. Then $x \in S^\perp$, so $qx \in q(S^\perp)$.
\end{proof}

\begin{prop}
\[
S + S^\perp = V.
\]
\end{prop}

\begin{proof}
We have
\[
q(S + S^\perp)
= q(S) + q(S^\perp)
= q(S) + q(S)^\perp
= V / \Ker b
\]
as $\hat q$ is non-degenerate. As $q$ is surjective, this implies the result.
\end{proof}

\begin{prop}
\[
S \cap S^\perp = S \cap \Ker b.
\]
\end{prop}

\begin{proof}
We have
\[
q(S \cap S^\perp)
= q(S) \cap q(S^\perp)
= q(S) \cap q(S)^\perp
= 0
\]
as $\hat q$ is non-degenerate,
so $S \cap S^\perp \subset \Ker q = \Ker b$. Conversely, $\Ker b \subset S^\perp$, so $S \cap \Ker b \subset S \cap S^\perp$.
\end{proof}

If $f : S \to V$ is a linear morphism and $b_V$ is a Hermitian form on $V$, then
we get an induced Hermitian form $b_S := f^*b_V$ on $S$. We can ask whether the
same happens when we have a morphism $f : V \to Q$ and a form $b_V$. Clearly we
can only expect something on the image of $f$, so we should take it to be
surjective. When the form on $b_V$ is positive-definite, we can define a form on
$Q$ by using the injection $0 \to Q^* \to V^*$ and the dual Hermitian forms. In
general, we can still get a Hermitian form on the quotient space by using
orthogonal complements.

\begin{theo}
Let $q : V \to Q$ be a surjective morphism. A Hermitian form $b_V$ on $V$ induces a Hermitian form $b_Q$ on $Q$.
\end{theo}

\begin{proof}
Let $S = \Ker q$, and consider $S^\perp \subset V$. We have a short exact sequence
\[
\begin{tikzcd}
0 \ar[r] &
S \cap \Ker b_V \ar[r] &
S^\perp \ar[r,"q"] &
Q \ar[r] &
0
\end{tikzcd}
\]
because $S^\perp \cap \Ker q = S^\perp \cap S = S \cap \Ker b_V$, and $q$ restricted to $S^\perp$ is still surjective. We define
\[
b_Q(qx, \ov{qy})
= b_V(x, \ov y)
\]
for $x, y \in S^\perp$. This is well-defined by the above, as any other element $x'$ that maps to $qx$ satisfies $x' - x \in S \cap \Ker b_V \subset \Ker b_V$.
\end{proof}

\begin{prop}
Let $q : V \to Q$ be surjective, and let $b_V$ and $b_Q$ be as above. Then $q$ admits an adjoint $q^\dagger : Q \to V$.
\end{prop}

\begin{proof}
We need to show that $q(\Ker b_V) \subset \Ker b_Q$. If $x, y$ are elements of $V$, we have by definition
\[
b_Q(qx, \ov{qy}) = b_V(x, \ov y).
\]
If $x \in \Ker b_V$, this implies that $b_Q(qx, \ov{qy}) = b_V(x, \ov y) = 0$ for all $y \in V$, so $qx \in \Ker b_Q$.
\end{proof}

\subsection*{Induced forms on tensor algebras}
\label{sec:induced-forms-tensor}

If $V$ is a vector space with a Hermitian form $b_V$, we get induced Hermitian forms on any vector space derived from the tensor algebra of $V$ if we restrict ourselves to spaces that only involve tensor powers of $V$ and not $V^*$. Handling $V^*$, or more generally $\Hom(V,W)$, is more delicate when the forms are possibly degenerate.

If $V$ and $W$ are vector spaces with Hermitian forms $b_V$ and $b_W$, we'd like to use those forms to define a Hermitian form on $\Hom(V,W)$. When the form $b_V$ is non-degenerate, every linear morphism $f : V \to W$ admits a unique adjoint $f^\dagger : W \to V$, and the Fr\"obenius form on $\Hom(V,W)$ is
\[
b(f, \ov g) = \tr(g^\dagger f).
\]
In general, not every morphism $f : V \to W$ admits an adjoint, and the ones that do may not admit unique adjoints. We could try to define a Hermitian form on the subspace of morphisms that admit adjoints by $\tr(g^\dagger f)$, but as adjoints may not be unique this is not well-defined.

The way around this difficulty is again to pass to the quotient spaces. A morphism $f : V \to W$ admits an adjoint if and only if $f(\Ker b_V) \subset \Ker b_W$. In the latter case, there exists a unique morphism $\hat f : V / \Ker b_V \to W / \Ker b_W$ such that the diagram
\[
\begin{tikzcd}
V \ar[d,"\pi_V"] \ar[r,"f"] & W \ar[d,"\pi_W"]
\\
V / \Ker b_V \ar[r,"\hat f"] & W / \Ker b_W
\end{tikzcd}
\]
commutes; it is defined by $\hat f([v]) = \pi_W(f(v))$.

\begin{defi}
Let $(V,b_V)$ and $(W,b_W)$ be Hermitian spaces. We define
\[
b_{\Hom(V,W)}(f, \ov g) := \tr((\hat g)^\dagger \hat f)
\]
on the subspace of $\Hom(V,W)$ of morphisms that admit adjoints.
\end{defi}

If we denote by $\Hom_b(V,W) \subset \Hom(V,W)$ the subspace of morphisms that admit adjoints, the definition says that we have a surjective morphism
\[
\Hom_b(V,W) \to \Hom(V/\Ker b_V, W/\Ker b_W)
\]
and that we define a Hermitian form upstairs by pullback of the one downstairs, which exists because the Hermitian forms there are non-degenerate.

The special case $V^* = \Hom(V, \kk C)$, where $\kk C$ is equipped with the
standard inner product, is interesting.
Here an element $\lambda \in V^*$ admits an adjoint if and only if $\lambda(\Ker b_V)
= 0$, that is, if $\Ker b_V \subset \Ker \lambda$.
Note however that the dimension of the space of dual vectors that admit adjoints
is $\dim V - \dim \Ker b_V$.
Meanwhile, the dimension of the space of morphisms $V \to W$ that admit adjoints
is $\dim V \dim W - \dim \Ker b_V (\dim W - \dim \Ker b_W)$.
We then have
$$
\dim \Hom_b(V, W)
- \dim V^*_b \dim W
= \dim \Ker b_v \dim \Ker b_W,
$$
which is generally positive, so the canonical map $V_b^* \otimes W \to \Hom_b(V,
W)$ is generally only injective and not bijective like in the non-degenerate
case.
The problem is that the image of any element on the left-hand side will be a
morphism $f$ that satisfies $f(\Ker b_V) = 0$, while there are more elements on
the right-hand side.

\section{Chern connections}
\label{sec:degenerate-chern-connections}

Let $E \to X$ be a holomorphic vector bundle with a smooth Hermitian form $b$. We say that a connection $D$ on $E$ is \emph{compatible with $b$} if
\[
d b(s, \ov t) = b(Ds, \ov t) + b(s, \ov{Dt})
\]
for all sections $s, t$ of $E$. By composing a connection with the projection onto $(1,0)$- or $(0,1)$-forms, we get the $(1,0)$- and $(0,1)$-parts of the connection.

\begin{prop}
Let $E \to X$ be a holomorphic vector bundle and let $b$ be a smooth Hermitian form on $E$. Then there exists a connection $D$ on $E$ that's compatible with $b$ and has $D^{0,1} = \bar\partial$.
\end{prop}

\begin{proof}
The plan is to construct a connection locally on trivializing neighborhoods of $E$ and to glue the pieces together with a partition of unity.

Let $s$ be a smooth section of $E$ over a neighborhood $U$. We want to solve the equation
\begin{equation}
\label{eq:conn}
\partial_\alpha b(s, \ov t)
= b(A(\alpha), \ov t),
\end{equation}
where $\alpha \in \cc C^\infty(U, T_X)$, $t \in \cc C^\infty(U, E)$ and $A \in \cc A^{1,0}(U, E)$, where the last space is the one of smooth $(1,0)$-forms with values in $E$. Consider the short exact sequence
\[
\begin{tikzcd}
0 \ar[r] &
\Ker b \ar[r] &
E \ar[r,"b"] &
\im b \ar[r] &
0,
\end{tikzcd}
\]
and tensor it by the vector bundle $\Omega^1_X$ to obtain
\[
\begin{tikzcd}
0 \ar[r] &
\Omega^1_X \otimes \Ker b \ar[r] &
\Omega^1_X \otimes E \ar[r,"\id \otimes b"] &
\Omega^1_X \otimes \im b \ar[r] &
0.
\end{tikzcd}
\]
We claim that $\alpha \otimes t \mapsto \partial_\alpha b(s, \ov t)$ is a smooth section of $\Omega^1_X \otimes \im b$. By the above, this is true if it is zero on the sheaf generated by sections of the form $\alpha \otimes t$, where $t \in \Ker b$. But for such a section we have $b(s, \ov t) = 0$ on $U$, so $\partial_\alpha b(s, \ov t) = 0$. By taking a small enough neighborhood so the sheaf morphism above is surjective, or noting that the sheaves are ones of smooth modules which are fine, we obtain a section $A_s \in \cc C^\infty(U, \Omega_X^1 \otimes E)$ such that \eqref{eq:conn} holds for any $\alpha$ and $t$.

Suppose now that there exists a holomorphic frame $(e_1, \ldots, e_r)$ of $E$ over $U$. Such a frame exists over a small enough neighborhood around any point. Let $A_1, \ldots, A_r$ be the sections we just found for $e_1, \ldots, e_r$. Any section $s \in \cc C^\infty(U, E)$ can be written as $s = \sum_j f_j e_j$ for smooth functions $f_j$ on $U$. We define
\[
D s := \sum_{j=1}^r df_j \otimes e_j + f_j A_j
\]
on $U$. We claim that $D$ is a connection on $E$ over $U$ that is compatible with $b$ and whose $(0,1)$-part is $\bar\partial$.

This object $D$ is a $1$-form with values in $E$ by construction of the $A_j$. If $s$ is a section of $E$ and $f$ a smooth function, we write $s = \sum_j f_j e_j$. Then
\begin{align*}
D(fs)
&= \sum d(f f_j) \otimes e_j + f f_j A_j
\\
&= \sum df \otimes f_j e_j + f df \otimes e_j + f f_j A_j
= df \otimes s + f Ds,
\end{align*}
so $D$ is a connection on $E$. The forms $A_j$ are $(1,0)$-forms by construction, so the $(0,1)$-part of $Ds$ is $\sum_j \bar\partial f_j e_j = \bar\partial s$.

Finally, let $f$ and $g$ be smooth functions. We have
\begin{align*}
d b(fe_j, \ov{ge_k})
&= d\bigl(f \ov g b(e_j, \ov e_k) \bigr)
\\
&= df \cdot \ov g b(e_j, \ov e_k)
+ f dg \cdot b(e_j, \ov e_k)
+ \partial b(e_j, \ov e_k)
+ \ov{\partial b(e_k, \ov e_j)}
\\
&= b(df \otimes e_j, \ov{ge_k})
+ b(f e_j, \ov{dg \otimes e_k})
+ b(A_j, \ov e_k)
+ b(e_j, \ov{A_k})
\\
&= b(D(f e_j), \ov{ge_k}) + b(f e_j, \ov{D(g e_k)}).
\end{align*}
For smooth sections $s = \sum f_j e_j$ and $t = \sum g_k e_k$ we deduce from the above and linearity that $D$ is compatible with $b$ over the neighborhood $U$.

We now cover $X$ by neighborhoods $U_\mu$ as above and construct connections $D_\mu$ on each $U_\mu$. We then take a partition of unity $(\phi_\mu)$ relative to this covering, and define a connection $D$ on $E$ by setting
\[
D s = \sum_\mu D_\mu(\phi_\mu s).
\]
This defines a connection on $E$ that's compatible with $b$ and whose $(0,1)$-part is $\bar\partial$.
\end{proof}

We will call a connection $D$ that's compatible with $b$ and has $D^{0,1} = \bar\partial$ a \emph{Chern connection} of $b$. If $b$ has a nontrivial kernel, such a connection is not unique.

\begin{prop}
Let $D_1$ and $D_2$ be Chern connections of $b$. Then $D_1 - D_2$ is an element of $\cc A^{1,0}(\Hom(E, \Ker b))$.
\end{prop}

\begin{proof}
It is well-known that the difference of two connections is an element of $\cc A^1(\End E)$. Both connections here have the same $(0,1)$-part, so their difference is a $(1,0)$-form. If $s$ and $t$ are holomorphic sections, we have
\[
b((D'_1 - D'_2)s, \ov t)
= \partial b(s, \ov t) - \partial b(s, \ov t) = 0,
\]
so $D_1 - D_2 = D_1' - D_2'$ takes values in $\Ker b$.
\end{proof}

\begin{prop}
If $s$ is a section of $\Ker b$ and $D$ is a Chern connection of $E$, then $Ds$ takes values in $\Ker b$.
\end{prop}

\begin{proof}
If $s$ is a section of $\Ker b$, then
\[
0
= \partial b(s, \ov t)
= b(Ds, \ov t)
\]
for all holomorphic sections $t$ of $E$, so $D s \in \cc A^1(\Ker b)$.
\end{proof}

\begin{theo}
Let $(E, b) \to X$ be a holomorphic vector bundle with a Hermitian form, and let $D$ be a Chern connection of $E$. The curvature tensor
\[
R(\alpha,\ov\beta,s, \ov t)
= b\bigl(\tfrac i2 D^2_{\alpha\ov\beta}s, \ov t\bigr)
\]
is independent of the choice of $D$.
\end{theo}

\begin{proof}
Let $D_1$ and $D_2$ be Chern connections of $b$. Then $D_1 = D_2 + A$, where $A \in \Hom(E, \Ker b)$, and
\[
D_1^2
= (D_2 + A)^2
= D_2^2 + A D_2 + D_2 A + A^2.
\]
The difference $D_1^2 - D_2^2$ thus takes values in $\Ker b$.
\end{proof}

\begin{prop}
Let $(E,b_E) \to X$ and $(F,b_F) \to X$ be holomorphic vector bundles with Hermitian forms and let $D_E$ and $D_F$ be Chern connections of $E$ and $F$. Let $f : E \to F$ be a smooth bundle morphism that admits an adjoint $f^\dagger$. Then $D_{\Hom(E,F)}f$ admits an adjoint and $(D_{\Hom(E,F)}f)^\dagger$ and $D_{\Hom(F,E)}f^\dagger$ are both its adjoints.
\end{prop}

\begin{proof}
The morphism $f$ admits an adjoint if and only if $f(\Ker b_E) \subset \Ker b_F$. Let $s$ be a section of $\Ker b_E$, so $f(s)$ is a section of $\Ker b_F$. Then
\begin{align*}
0 = d b_F(f(s), \ov t)
&= b_F(D_F(f(s)), \ov t) + b_F(f(s), \ov{Dt})
\\
&= b_F(D_F(f(s)), \ov t)
\\
&= b_F(D_{\Hom(E,F)}f(s), \ov t) + b_F(f(D_Es), \ov t)
\\
&= b_F(D_{\Hom(E,F)}f(s), \ov t)
\end{align*}
for all sections $t$, as $D_Es$ is a one-form with values in $\Ker b_E$. Then $D_{\Hom(E,F)} f(s) \in \cc A^1(\Ker b_F)$ for all sections $s$ of $\Ker b_E$, so $D_{\Hom(E,F)}f$ admits an adjoint.

We have $b_F(f(s), \ov{t}) = b_E(s, \ov{f^\dagger t})$ for all sections $s,t$. Taking the exterior derivative of both sides we get
\[
b_F(D_F(f(s)), \ov t) + b_F(f(s), \ov{D_F t})
= b_E(D_E s, \ov {f^\dagger t}) + b_E(s, \ov{D_E(f^\dagger t)}).
\]
The left-hand side equals
\[
b_F(D_{\Hom(E,F)}f(s), \ov t)
+ b_F(f(D_Es), \ov t)
+ b_E(s, \ov{f^\dagger(D_F t)})
\]
while the right-hand side equals
\[
b_E(f(D_E s), \ov t)
+ b_E(s, \ov{D_{\Hom(F,E)}f^\dagger(t)})
+ b_E(s, \ov{f^\dagger(D_F t)})
\]
so after cancelling terms we have
\[
b_E(s, \ov{D_{\Hom(F,E)}f^\dagger(t)})
= b_F(D_{\Hom(E,F)}f(s), \ov t)
= b_E(s, \ov{(D_{\Hom(E,F)}f)^\dagger(t)})
\]
for all sections $s,t$.
\end{proof}

When $b_E$ is non-degenerate morphisms from $E$ to $F$ admit a unique adjoint, so in that case $D_{\Hom(F,E)}f^\dagger = (D_{\Hom(E,F)}f)^\dagger$. Note that taking the adjoints of forms with values in a bundle conjugates the form part.

\begin{prop}
Let $f : X \to Y$ be a holomorphic morphism and $E \to Y$ be a holomorphic vector bundle.
Let $b$ be a Hermitian form on $E$ with a Chern connection $D$.
Then $f^{*}D$ is a Chern connection of $f^{*}b$ on $f^{*}E \to X$.
\end{prop}

\begin{proof}
If $s$ and $t$ are holomorphic sections of $E$ and $\xi$ a tangent vector at a point on $X$ we have
\[
\partial_{\xi} f^{*} b(f^{*}s, \overline{f^{*}t})
= f^{*}(\partial_{f_{*}\xi} b(s, \overline t))
= f^{*}(b(D_{f_{*}\xi}s, \overline{t}))
= f^{*}b(f^{*}D_{\xi} f^{*}s, \overline{f^{*}t}).
\]
As $f^{*}E$ is generated by sections of the form $f^{*}s$, this completes the proof.
\end{proof}

\begin{prop}
Let $X$ be a complex manifold and let $b$ be a Hermitian form on $T_X$.
Let $\omega = -\operatorname{Im} b$ be the associated K\"ahler form.
Then $d \omega = 0$ if and only if the torsion tensor $\tau(\alpha,\beta) = D_\alpha \beta - D_\beta \alpha - [\alpha, \beta]$ of $b$ takes values in $\Ker b$.
\end{prop}

\begin{proof}
For holomorphic tangent fields we have
\[
\partial b (\alpha, \beta, \ov \gamma)
= \partial_\alpha b(\beta, \ov\gamma)
- \partial_\beta b(\alpha, \ov\gamma)
- b([\alpha,\beta], \ov\gamma)
= b(\tau(\alpha,\beta), \ov\gamma)
\]
and $\partial b = 0$ if and only if $\partial \omega = 0$.
\end{proof}

\subsection*{A word on notation}

In a lot of situations we care about, we now only have equality up to things that take values in the kernel of a Hermitian form. For example, if $f : E \to F$ is a Hermitian morphism, it admits an adjoint $f^\dagger$ and we have
\[
b_E(f^\dagger f s, \ov t)
= b_F(f s, \ov{f t})
= b_E(s, \ov t)
\]
for all sections $s$ and $t$,
so $f^\dagger f - \id_E \in \Hom(E, \Ker b_E)$. When working on vector spaces we could deal with this by quotienting out the kernel. If we try the same trick here, we run into the fact that the rank of the sheaf $\Ker b \to X$ may not be constant. We want to do differential geometry to the sections of this sheaf, so this is a problem.

This difficulty is a mirage as it can be taken care of by notation. We define an equivalence relation on sections by saying that two sections are equivalent if their difference is in the kernel of a Hermitian form of interest:

\begin{defi}
If $(E, b) \to X$ is a holomorphic vector bundle with a Hermitian form $b$, and $s, t \in \cc A^k(E)$, we write
\(
s \sim t
\)
if $s - t \in \cc A^k(\Ker b)$.
\end{defi}

In general it's problematic to define differential forms with values in a sheaf, but here it's hopefully clear what $\cc A^k(\Ker b)$ means as $\Ker b$ is a subsheaf of the sheaf of sections of a vector bundle.

Using this notation we can rephrase some of our earlier results on degenerate forms:
\begin{itemize}
\item
If $f : E \to F$ is a Hermitian morphism and $f^\dagger$ an adjoint, then $f^\dagger f \sim \id_E$.

\item
If $f : E \to F$ is a morphism and $f_1^\dagger$ and $f_2^\dagger$ are adjoints of $f$, then $f_1^\dagger \sim f_2^\dagger$.

\item
If $f : E \to F$ is a morphism that admits an adjoint, then $D_{\Hom(E,F)}$ has adjoints $(D_{\Hom(E,F)}f)^\dagger \sim D_{\Hom(F,E)}f^\dagger$.

\item
If $D$ is a Chern connection of $(E,b)$ and $s \sim t$ then $Ds \sim Dt$.

\item
If $D_1$ and $D_2$ are Chern connections of $(E,b)$, then $D_1s \sim D_2s$ for all sections $s$.
\end{itemize}
As adjoints and connections behave well with respect to this relation, we can prove the result we want by essentially the same arguments as in the case where the Hermitian forms are positive-definite; we simply replace ``$=$'' in those proofs by ``$\sim$'' and note that everything else works as before.

An example of this notation's usefulness is the following result, which is annoying to state and use without it. For Hermitian metrics, the $(2,0)$-part of the curvature form is zero. For our Chern connections, the same is not necessarily true; we can't say what happens in the kernel of the form.

\begin{prop}
Let $(E, b) \to X$ be a Hermitian bundle, and let $D$ be a Chern connection for $b$. Then $(D')^2 \sim 0$.
\end{prop}

\begin{proof}
We have
\[
\partial b(s, \ov t)
= b(D's, \ov t) + b(s, \ov{\bar\partial t})
\]
so
\[
0 = \partial^2 b(s, \ov t)
= b((D')^2s, \ov t)
- b(D's, \ov{\bar\partial t})
+ b(D's, \ov{\bar\partial t})
+ b(s, \ov{\bar\partial^2 t})
= b((D')^2s, \ov t)
\]
for all sections $s$ and $t$.
\end{proof}

\subsection*{Sub- and quotient bundles}

Let
\[
\begin{tikzcd}
0 \ar[r] &
S \ar[r,"j"] &
E \ar[r,"q"] &
Q \ar[r] &
0
\end{tikzcd}
\]
be a short exact sequence of holomorphic bundles over $X$, and let $b_E$ be a smooth Hermitian form on $E$. It induces Hermitian forms $b_S$ and $b_Q$ on $S$ and $Q$ as we have seen. Both morphisms $j$ and $q$ also admit adjoints $j^\dagger: E \to S$ and $q^\dagger : Q \to E$. We denote Chern connections of $S$, $E$ and $Q$ by $D_S$, $D_E$ and $D_Q$.

\begin{defi}
The \emph{second fundamental form} of $S$ in $E$ is
\[
\snd(s) := q(D_E(js) - jD_S(s)).
\]
\end{defi}

\begin{prop}
The second fundamental form is an element of $\cc A^{1,0}(\Hom(S,Q))$.
We have
\[
\snd(s)
= q(D_{\Hom(S,E)}j(s))
= q(D_E(js))
\]
for sections $s$ of $S$.
\end{prop}

\begin{proof}
As $qj = 0$ it is clear that $b(s) = q(D_E(js))$. We also have
\[
D_E(js) = D_{\Hom(S,E)}(j)(s) + j(D_Ss),
\]
so $b(s) = q(D_{\Hom(S,E)}(j)(s))$.

As $j$ is holomorphic, we have $D_{\Hom(S,E)}j = D'_{\Hom(S,E)}j$ so $\snd$ has no $(0,1)$-part. It is also clearly $\cc C^\infty$-linear in its tensor field variable. If $f$ is a smooth function, we have
\[
D_E(j(fs))
= df \otimes js + f D_E(js)
\]
so $\snd(fs) = q(D_E(j(fs))) = q(f D_E(js)) = f\snd(s)$.
\end{proof}

\begin{prop}
The second fundamental form admits an adjoint.
\end{prop}

\begin{proof}
We have to show that $\snd(s) \in \Ker b_Q$ when $s$ is a section of $\Ker b_S$. If $s$ is such a section, then $js$ is a section of $\Ker b_E$, so $D_E(js)$ is a section of $\Ker b_E$. Then $q(D_E(js))$ is a section of $\Ker b_Q$ by the construction of $b_Q$.
\end{proof}

Our goal is to calculate the curvature forms of the sub- and quotient bundles. The first stop on the way is a wildly useful list of formulas from Demailly~{{\cite[Theorem~14.3]{demailly-complex}}}.

\begin{prop}
\label{prop:seq-formulas}
\begin{alignat*}{2}
D'_{\Hom(S,E)}j &\sim q^\dagger \circ \snd,
&
\qquad
\bar\partial j &= 0,
\\
D'_{\Hom(E, Q)} q &\sim - \snd \circ j^\dagger,
&
\bar\partial q &= 0,
\\
D'_{\Hom(E,S)} j^\dagger &\sim 0,
&
\bar\partial j^\dagger &\sim b^\dagger \circ q,
\\
D'_{\Hom(Q,E)} q^\dagger &\sim 0,
&
\bar\partial q^\dagger &\sim - j \circ \snd^\dagger,
\\
D'_{\Hom(S,Q)} \snd &\sim 0,
&
\bar\partial \snd^\dagger &\sim 0.
\end{alignat*}
\end{prop}

\begin{proof}
Let $s$ be a section of $S$. Then
\[
j(D_S s) + q^\dagger \snd(s)
\sim D_E(js)
= D_{\Hom(S,E)}j (s) + j(D_S s),
\]
where the first relation is by definition of $D_S$ and $\snd$, so
\[
D_{\Hom(S,E)}j \sim q^\dagger \circ \snd.
\]
The morphism $j$ is holomorphic, so $\bar\partial j = 0$. This proves the first line. The third line follows from the first by taking adjoints.

For the second line, the morphism $q$ is also holomorphic, so $\bar\partial q = 0$. Taking adjoints, we get that $D'_{\Hom(Q,E)} q^\dagger \sim 0$.
We have $\id_E \sim j \circ j^\dagger + q^\dagger \circ q$, so
\begin{align*}
0
&= D'_{\End E} \id_E
\\
&\sim D'_{\Hom(S,E)} j \circ j^\dagger + j \circ D'_{\Hom(E,S)}j^\dagger
+ D'_{\Hom(Q,E)}q^\dagger \circ q + q^\dagger D'_{\Hom(E,Q)} q
\\
&\sim q^\dagger \circ \snd \circ j^\dagger + q^\dagger \circ D'_{\Hom(E,Q)}q.
\end{align*}
Applying $q$, we conclude that $D'_{\Hom(E,Q)}q \sim - \snd \circ j^\dagger$. This proves the second line. Taking adjoints we get the last piece of the fourth line.

Finally we note that
\begin{align*}
D'_{\Hom(S,Q)} \snd
&\sim D'_{\Hom(S,Q)} (q D'_{\Hom(S,E)}j)
\\
&\sim - \snd \circ j^\dagger \circ q^\dagger \circ \snd
+ q^\dagger (D'_{\Hom(S,Q)})^2 j
\sim 0
\end{align*}
as we're dealing with curvature tensors of Hermitian forms which are of type $(1,1)$ up to things that are $\sim 0$. The statement about $\bar\partial \snd^\dagger$ follows by taking adjoints.
\end{proof}

\begin{prop}
Under the smooth splitting $E \to S \oplus Q$ defined by $s \mapsto j^\dagger s \oplus qs$ the Chern connection and curvature form of $E$ are
\begin{align*}
D_E s &\sim
\begin{pmatrix}
j^\dagger & q
\end{pmatrix}^\dagger
\begin{pmatrix}
D_S & - \snd^\dagger
\\
\snd & D_Q
\end{pmatrix}
\begin{pmatrix}
j^\dagger \\ q
\end{pmatrix}
(s),
\\
D_E^2 s &\sim
\begin{pmatrix}
j^\dagger & q
\end{pmatrix}^\dagger
\begin{pmatrix}
D^2_S - \snd^\dagger \wedge \snd & -D'_{\Hom(Q,S)} \snd^\dagger
\\
\bar\partial \snd & D^2_Q - \snd \wedge \snd^\dagger
\end{pmatrix}
\begin{pmatrix}
  j^\dagger \\ q
\end{pmatrix}(s).
\end{align*}
\end{prop}

\begin{proof}
Let $s$ be a section of $E$. We write
\[
s \sim j (j^\dagger s) + q^\dagger( qs).
\]
Then
\[
D_E s
\sim q^\dagger \circ \snd (j^\dagger s)
+ j \circ D_S( j^\dagger s)
- j \circ \snd^\dagger (qs)
+ q^\dagger \circ D_Q (qs).
\]
This proves the first statement.
For each of the terms here, we have
\begin{align*}
D_E(q^\dagger \circ \snd (j^\dagger s))
&\sim -j \circ \snd^\dagger \circ \snd (j^\dagger s)
+ q^\dagger \circ \bar\partial \snd (j^\dagger s)
- q^\dagger \circ \snd \circ D_S(j^\dagger s),
\\
D_E(j \circ D_S( j^\dagger s))
&\sim q^\dagger \circ \snd \circ D_S(j^\dagger s)
+ j \circ D_S^2 (j^\dagger s),
\\
D_E(-j \circ \snd^\dagger (qs))
&\sim - q^\dagger \circ \snd \circ \snd^\dagger (qs)
- j \circ D'_{\Hom(Q,S)}\snd^\dagger (qs)
+ j \circ \snd^\dagger \circ D_Q(qs),
\\
D_E(q^\dagger \circ D_Q (qs))
&\sim -j \circ \snd^\dagger \circ D_Q(qs)
+ q^\dagger \circ D_Q^2 (qs).
\end{align*}
Grouping these together by pre- and postfix morphisms, we get
\begin{align*}
D_E^2 s
&\sim j \bigl( D_S^2 - \snd^\dagger \circ \snd \bigr) j^\dagger s
- j \bigl( D'_{\Hom(Q,S)} \snd^\dagger \bigr) qs
\\
&\qquad
+ q^\dagger \bigl( \bar\partial \snd \bigr) j^\dagger s
+ q^\dagger \bigl( D^2_Q - \snd \circ \snd^\dagger \bigr) qs
\\
&=
\begin{pmatrix}
j^\dagger & q
\end{pmatrix}^\dagger
\begin{pmatrix}
D^2_S - \snd^\dagger \wedge \snd & -D'_{\Hom(Q,S)} \snd^\dagger
\\
\bar\partial \snd & D^2_Q - \snd \wedge \snd^\dagger
\end{pmatrix}
\begin{pmatrix}
  j^\dagger \\ q
\end{pmatrix}(s),
\end{align*}
which completes the proof.
\end{proof}

\begin{coro}[Codazzi--Griffiths equations]
The curvature tensors of $S$ and $Q$ satisfy
\begin{align*}
R_S(\alpha, \ov\beta, s, \ov t)
&= R_E(\alpha, \ov\beta, js, \ov{jt})
- b_Q(\snd(\alpha, s), \ov{\snd(\beta, t)}),
\\
R_Q(\alpha, \ov\beta, s, \ov t)
&= R_E(\alpha, \ov\beta, q^\dagger s, \ov{q^\dagger t})
+ b_S(\snd^\dagger(\ov\beta, s), \ov{\snd^\dagger(\ov\alpha, t)}).
\end{align*}
\end{coro}

\begin{coro}
Let $E \to X$ be a holomorphic vector bundle with two Hermitian forms $b_1$ and $b_2$. Assume that $h := b_1 + b_2$ is positive-definite. The curvature tensor of $h$ is
$$
\displaylines{
R_h(\alpha, \ov\beta, s, \ov t)
= R_{b_1}(\alpha, \ov\beta, s, \ov t)
+ R_{b_2}(\alpha, \ov\beta, s, \ov t)
- q(\sigma(\alpha) s, \ov{\sigma(\beta) t}),
}
$$
where the second fundamental form is
$$
\sigma(\alpha) s = (D_{b_1,\alpha} - D_{b_2,\alpha}) s
$$
and the Hermitian form on the ``quotient'' bundle $E$ is
$$
q(s, \ov t)
= b_1(h^{-1}b_2 s, \ov{h^{-1}b_2 t}) + b_2(h^{-1}b_1 s, \ov{h^{-1}b_1 t}).
$$
\end{coro}

\begin{proof}
We have a short exact sequence
\[
  0 \longrightarrow E \longrightarrow E \oplus E \longrightarrow E \longrightarrow 0,
\]
where the inclusion is $j(s) = s \oplus s$, the quotient map is $\pi(s \oplus t) = s - t$, and the Hermitian metrics on the bundles are, in order, $h_{1} + h_{2}$, $h_{1} \oplus h_{2}$, and the ``quotient'' metric.
The Chern connection of the direct sum is $D_{1} \oplus D_{2}$ and its curvature tensor is $R_{1} \oplus R_{2}$.
The second fundamental form is
\[
  \sigma(\alpha) s
  = \pi(D_{1,\alpha} s \oplus D_{2, \alpha} s)
  = D_{1,\alpha} s - D_{2, \alpha} s.
\]
To compute the quotient form we
first  note that the adjoint of $j : E \to E \oplus E$ is defined by
\[
h(j^\dagger( s \oplus t), \ov r)
= h( s \oplus t, \ov{r \oplus r})
= b_1(s, \ov r) + b_2(t, \ov r)
\]
so
\[
j^\dagger(s \oplus t)
= h^{-1}(b_1(s) + b_2(t)).
\]
Then $s \oplus t = jj^\dagger (s\oplus t) + \pi^\dagger \pi(s \oplus t)$, so we find that
\[
\pi^\dagger \pi(s \oplus t)
= \bigl(
s - h^{-1}(b_1(s) + b_2(t))
\bigr) \oplus
\bigl(
t - h^{-1}(b_1(s) + b_2(t))
\bigr).
\]
We write $s = h^{-1}(b_1 + b_2)(s)$ and get
\[
s - h^{-1}(b_1(s) + b_2(t))
= h^{-1}b_2(s - t).
\]
Similarly we get
\[
t - h^{-1}(b_1(s) + b_2(t))
= -h^{-1}b_1(s - t).
\]
Then
\[
\pi^\dagger \pi(s \oplus t)
= \bigl(
h^{-1}b_2(s - t)
\bigr)
\oplus
\bigl(
-h^{-1}b_1(s - t)
\bigr).
\]
In particular, the adjoint of $\pi$ is
\[
\pi^\dagger(s)
=\bigl(
h^{-1}b_2(s)
\bigr)
\oplus
\bigl(
-h^{-1}b_1(s)
\bigr).
\]
The inner product on the quotient is the pullback of $h$ by $\pi^\dagger$, which completes the calculation.
\end{proof}

\section{Positive holomorphic sectional curvature}
\label{sec:grassmannian-bundles}

Our objective in this section is to prove the following technical theorem and apply it to things relevant to our interests.
This theorem was first proved by Chaturvedi and Heier~\cite[Theorem~1.1]{chaturvedi2020hermitian} by classical calculations in local coordinates.
We hope our (mostly) coordinate-free proof is also of interest.

\begin{theo}
\label{thm:holomorphic-sectional-positive}
Let $\pi : X \to (B,h_{B})$ be a holomorphic fibration over a compact Hermitian manifold.
Suppose $h_{B}$ has positive holomorphic sectional curvature.
Suppose there exists a Hermitian form $h_{X/B}$ on $X$ whose restriction to each fiber is a Hermitian metric that has positive holomorphic sectional curvature.
Then there exists a Hermitian metric on $X$ that has positive holomorphic sectional curvature, and if $h_{B}$ is K\"ahler and the K\"ahler form of $h_{X/B}$ is closed then this metric is also K\"ahler.
\end{theo}

\begin{proof}
We consider the Hermitian forms $h_{\lambda} = h_{X/B} + e^{\lambda} \pi^{*} h_{B}$ for real $\lambda$.
First off, it's clear that if $h_{B}$ is K\"ahler and the K\"ahler form of $h_{X/B}$ is closed, then $h_{\lambda}$ is K\"ahler.

As the base $B$ is compact, $h_{X/B}$ is positive on $T_{X/B}$, and $\pi^{*}h_{B}$ is positive on $\pi^{*}T_{B}$, we see that $h_{\lambda}$ is an honest Hermitian metric for all $\lambda \gg 0$.
Its curvature tensor is then
\[
  R_{\lambda}
  = R_{X/B} + e^{\lambda}\pi^{*}R_{B} - \sigma^{*}q_{\lambda}.
\]
We need to investigate what happens first in $T_{X/B}$ and then the rest of $T_{X}$.
The calculations can get a little messy, so we split them into a few technical lemmas.
The first one will help us say what happens on $T_{X/B}$, as the form $\pi^{*}h_{B}$ is zero there.

\begin{lemm}
Let $b_{1}$ and $b_{2}$ be Hermitian forms on a vector space $V$ such that $h := b_{1} + b_{2}$ is positive-definite.
Let $q$ be the associated ``quotient'' form.
Then $\Ker b_{j} \subset \Ker q$ for $j = 1,2$.
\end{lemm}

\begin{proof}
By symmetry it's enough to prove the statement for $j = 1$.
Recall that
\[
  q = (h^{-1}b_{2})^{*}b_{1} + (h^{-1}b_{1})^{*}b_{2}.
\]
We have $x = h^{-1}h(x) = h^{-1}b_{1}(x) + h^{-1}b_{2}(x)$, so if $x \in \Ker b_{1}$ then $x = h^{-1}b_{2}(x)$.
But then
\begin{equation*}
q(x) = (h^{-1}b_{2})^{*}b_{1}(x) = b_{1}(x) = 0.
\qedhere
\end{equation*}
\end{proof}

We can now deal with $T_{X/B}$, where $\pi^*R_B = 0$.
Let $(x,b)$ be a point in $X$, so that $x \in X_b$.
There exists a local normal frame of $T_{X_b}$ for $h_{X/B,b}$ centered at $x$.
In this frame, we have $D_{X/B} - \pi^* D_B = 0$ at $x$.
The only part of our curvature tensor that survives at $x$ is thus
$R_\lambda = R_{X/B}$,
which is positive, so $R_\lambda$ is positive on $T_{X/B}$.

Next we'll consider what happens on the rest of $T_{X}$.
As we don't know what $R_{X/B}$ does there, we have to increase the contribution of $\pi^{*}R_{B}$ and try to control the ``quotient'' form.
For that we need to know what happens to it as $\lambda$ grows.

\begin{lemm}
Let $0 \to S \to V \stackrel{\pi}{\to} Q \to 0$ be a short exact sequence of vector spaces.
Let $b_{1}$ be a Hermitian form on $V$ that is positive-definite on $S$ and $b_{2}$ be a Hermitian inner product on $Q$ such that $h_{\lambda} := b_{1} + e^{\lambda} \pi^{*} b_{2}$ is positive-definite for all $\lambda \geq 0$.
Let $q$ be the associated ``quotient'' form.
Then there exists a semipositive Hermitian form $q_{\infty}$ on $V$ that is positive-definite on $S$ such that $h_{\lambda} \to h_{\infty}$ as $\lambda \to \infty$.
\end{lemm}

\begin{proof}
As $h_{0}$ is positive-definite we can simultaneously diagonalize it and $b_{1}$.
Then $\pi^{*}b_{2} = h_{0} - b_{1}$ is also diagonal.
Denote the coefficients of $b_{1}$ by $x_{j}$ and the ones of $\pi^{*}b_{2}$ by $y_{j}$, then the coefficients of $h_{\lambda}$ are
\[
  \frac{e^{\lambda}x_{j} y_{j}}{x_{j} + e^{\lambda} y_{j}}
  = \frac{x_{j} y_{j}}{x_{j}e^{-\lambda} + y_{j}}
  \to
  \begin{cases}
    x_{j} & y_{j} \not= 0
    \\
    0 & y_{j} = 0
  \end{cases}
\]
as $\lambda \to \infty$.
The form this defines is semipositive on $V$ and positive-definite on~$S$.%
\end{proof}

The limit form is $(j j^{\dagger})^{*}b_{1}$, where $j^{\dagger}$ is the adjoint of $j$ with respect to $h$. There's probably an invariant proof of this somewhere, but we don't need it.

Now pick an auxiliary Hermitian metric on $X$ and define a compact sphere bundle
$S(T_X) \subset T_X$ to test positivity on.
Let
\begin{align*}
m_\lambda &=
\inf_{S(T_X)} R_{X/B}
- (D_{X/B} - \pi^* D_{B})^* q_{\lambda},
\\
m_\infty &=
\inf_{S(T_X)} R_{X/B}
- (D_{X/B} - \pi^* D_{B})^* q_{\infty}
\end{align*}
As $S(T_X)$ is compact, the infimum $m_\infty$ is attained at a point $v$.
If $v \in \Ker \pi^*$, then we're done as we know that $R_{X/B}$ is positive
there.
Otherwise $\pi^*(v) \not= 0$, so $m_2 := \pi^*R_B(v, \bar v, v, \bar v) > 0$.
We then pick $\lambda_0$ big enough so that $m_\infty + e^{\lambda_0} m_2 > 0$.
Now note that
$m_\lambda \to m_\infty$ as $\lambda \to \infty$ because $q_{\lambda} \to
q_{\infty}$.
For $\lambda > \lambda_0$ we then have
$$
m_\lambda + e^\lambda m_2
\geq m_\lambda + e^{\lambda_0} m_2
\to m_\infty  + e^{\lambda_0} m_2
> 0.
$$
For all large enough $\lambda$, it follows that $R_\lambda$ will be positive.
\end{proof}

As a corollaries, we get a family of theorems where a holomorphic family whose fibers have positive holomorphic sectional curvature also has positive holomorphic sectional curvature as soon as we can pick the metrics on the fibers to vary smoothly between them.
The trick is to find conditions that guarantee we can extend the fiberwise metrics to a single smooth Hermitian form.
It is tempting to guess that if the metrics on the fibers are K\"ahler then we also get a K\"ahler metric on the total space, but this is delicate.
Given a family of smoothly varying K\"ahler metrics on each fiber, there does not need to exist a closed Hermitian form on the total space that restricts to each of them.
Remember for example the Hopf surface, which is the total space of a fibration of projective lines over an elliptic curve; picking any Hermitian metric on the total space we get a family of K\"ahler metrics on the fibers, but if they came from a closed form on the total space the Hopf surface would be K\"ahler.

\subsection*{Grassmannian bundles}

Recall that a Grassmannian manifold of $k$-planes in $n$-dimensional space admits a unique K\"ahler--Einstein metric $\omega_{\Gr}$ such that
$\Ric \omega_{\Gr} = n \omega_{\Gr}$.
When we speak of ``the'' K\"ahler--Einstein metric on a Grassmannian, this is the one we mean.
This metric can be obtained in two ways:

\begin{itemize}
  \item It is the curvature form of the Hermitian metric on $\det K_{\Gr}$ induced by the choice of any Hermitian inner product on the underlying vector space $V$. This agrees with the metric obtained by pullback of the Fubini--Study metric by the Pl\"ucker embedding.

  \item It is the pullback of the Hermitian metric on $\Hom(\cc S, \cc Q)$ induced by a Hermitian inner product on the underlying vector space by the (holomorphic) second fundamental form, where $0 \to \cc S \to \underline {V} \to \cc Q \to 0$ is the short exact sequence of tautological sub- and quotient bundles of the trivial bundle with fiber $V$.
\end{itemize}

Either way, the holomorphic sectional curvature $H$ of this metric satisfies
\[
  \frac{2}{k^{2}} \leq H \leq 2
\]
and is thus positive.
It's easiest to calculate its holomorphic sectional curvature by going the $\Hom(\cc S, \cc Q)$ route, in case the reader wants to try it out.

We can now slightly extend the results of \'Alvarez and her collaborators on projectivized vector bundles by not requiring our Grassmannian bundle to come from a vector bundle.

\begin{prop}[relative Pl\"ucker embedding]
Let $\pi : X \to B$ be a holomorphic fibration over a connected base whose fibers are Grassmannian manifolds.
Then there exists a holomorphic embedding $j : X \to Y$ of holomorphic fibrations
\[
\begin{tikzcd}
  X \ar[d,"\pi"] \ar[r,"j"] & Y \ar[d,"\hat\pi"]
  \\
  B \ar[r,"\id"] & B,
\end{tikzcd}
\]
where the fibers of $Y \to B$ are projective spaces, and the fiberwise maps $j_{b} : X_{b} \to Y_{b}$ are Pl\"ucker embeddings.
\end{prop}

\begin{proof}
Let $(U_{\nu})$ be a covering of $B$ by connected and simply connected charts.
Then there are local holomorphic isomorphisms $\theta_{\nu} : X_{|U_{\nu}} \to U_{\nu} \times \Gr(k)$, where $\Gr(k)$ is the Grassmannian of $k$-planes in $n$-dimensional space~\cite{fischer1965lokal}.
We define
\[
  j_{\nu} : U_{\nu} \times \Gr(k) \to U_{\nu} \times \kk P^{N},
  \quad
  (x, \Span(z_{1}, \ldots, z_{k}))
  \mapsto
  (x, [z_{1} \wedge \cdots \wedge z_{k}]),
\]
where $N = \binom nk$,
to be the usual Pl\"ucker embedding.

The coordinate change maps $\theta_{\mu\nu} = \theta_{\mu}\theta_{\nu}^{-1}$ can be written as
\[
\theta_{\mu\nu} = \id \times f_{\mu\nu},
\]
where $f_{\mu\nu}$ is an automorphism of $\Gr(k)$.
But an automorphism of $\Gr(k)$ is exactly an automorphism of $\kk P^{N}$ that preserves decomposable $k$-vectors~\cite{cowen1989automorphisms}, so these induce automorphisms $\hat f_{\mu\nu}$ of $\kk P^{N}$. We define the holomorphic family $Y \to B$ by the products $U_{\nu} \times \kk P^{N}$ and the gluing maps $\id \times \hat f_{\mu\nu}$.

By construction and the shape of automorphisms of the Grassmannian, it follows that the maps $j_{\nu}$ glue together to define an embedding $j : X \to Y$.
\end{proof}

Above we only really used that the fibers of $X$ embed in a manifold $V$ and the automorphism group of the fibers is a subgroup of the automorphism group of $V$, in case the reader has a favorite set of fibrations they'd like to try this out on.
We'll call the embedding $j : X \to Y$ above the relative Pl\"ucker embedding.

\begin{prop}
Let $\pi : X \to B$ be a holomorphic fibration over a connected base whose fibers are Grassmannian manifolds.
If $B$ admits a metric with positive holomorphic sectional curvature, then so does $X$.
This metric is K\"ahler if the metric on $B$ is K\"ahler.
\end{prop}

\begin{proof}
Consider the relative Pl\"ucker embedding $j : X \to Y$ as above.
Kodaira~\cite[Section 4(V)]{kodaira-embedding} showed that there exists a closed Hermitian form on $Y$ whose restriction to any fiber is the Fubini--Study metric.
The pullback of this metric to $X$ gives a form $h_{X/B}$ on $X$ whose restriction to each fiber is the K\"ahler--Einstein metric on the fiber.
Theorem~\ref{thm:holomorphic-sectional-positive} now applies.
\end{proof}

\begin{coro}
Let $\pi : X \to B$ be a holomorphic fibration over a compact K\"ahler manifold whose fibers are flag manifolds.
Suppose $B$ admits a K\"ahler metric with positive holomorphic sectional curvature.
Then there exists a K\"ahler metric on $X$ that has positive holomorphic sectional curvature.
\end{coro}

\begin{proof}
Denote the fibers by $P(d_{0}, \ldots, d_{k})$.
The forgetful map $P(d_{0}, \ldots, d_{k}) \to P(d_{0}, \ldots, \hat{d_{j}}, \ldots, d_{k})$ is a holomorphic fibration whose fiber is a Grassmannian.
Similar to before, this extends to a morphism $f$ of holomorphic fibrations
\[
\begin{tikzcd}
  X \ar[r,"f"]\ar[d] & X_{1} \ar[d]
  \\
  B \ar[r] & B
\end{tikzcd}
\]
whose fibers are Grassmannians.
Continuing this process we end up with a sequence $X \to X_{1} \to \cdots \to X_{k} \to B$ of fibrations with Grassmannian fibers, and Theorem~\ref{thm:holomorphic-sectional-positive} applies to each step in turn.
\end{proof}

\bibliographystyle{plain}
\bibliography{main}

\begin{thebibliography}{10}

\bibitem{alvarez2016positive}
Angelynn~Reario \'Alvarez.
\newblock {\em On the positive holomorphic sectional curvature of projectivized
  vector bundles over compact complex manifolds}.
\newblock PhD thesis, 2016.

\bibitem{alvarez2018projectivized}
Angelynn~Reario \'Alvarez, Gordon Heier, and Fangyang Zheng.
\newblock On projectivized vector bundles and positive holomorphic sectional
  curvature.
\newblock {\em Proc. AMS}, 146(7):2877--2882, 2018.

\bibitem{bel1975degenerate}
IV~Bel'ko.
\newblock Degenerate riemannian metrics.
\newblock {\em Mathematical notes of the Academy of Sciences of the USSR},
  18(5):1046--1049, 1975.

\bibitem{yang2019hirzebruch}
Yang Bo and Fangyang Zheng.
\newblock Hirzebruch manifolds and positive holomorphic sectional curvature.
\newblock In {\em Ann. Institut Fourier}, volume~69, pages 2589--2634, 2019.

\bibitem{calabi-fibres-holomorphes}
Eugenio Calabi.
\newblock M\'etriques k\"ahl\'eriennes et fibr\'es holomorphes.
\newblock {\em Ann. sci. \'ENS}, 12(2):269--294, 1979.

\bibitem{chaturvedi2020hermitian}
Ananya Chaturvedi and Gordon Heier.
\newblock Hermitian metrics of positive holomorphic sectional curvature on
  fibrations.
\newblock {\em Mathematische Zeitschrift}, 295(1):349--364, 2020.

\bibitem{cowen1989automorphisms}
Michael~J Cowen.
\newblock Automorphisms of grassmannians.
\newblock {\em Proceedings of the American Mathematical Society},
  106(1):99--106, 1989.

\bibitem{demailly-complex}
Jean-Pierre Demailly.
\newblock {\em Complex analytic and differential geometry}.
\newblock 2021.

\bibitem{fischer1965lokal}
Wolfgang Fischer and Hans Grauert.
\newblock {\em Lokal-triviale Familien kompakter komplexer Mannigfaltigkeiten}.
\newblock Vandenhoeck \& Ruprecht, 1965.

\bibitem{griffiths1965hermitian}
Phillip~A Griffiths.
\newblock Hermitian differential geometry and the theory of positive and ample
  holomorphic vector bundles.
\newblock {\em J. Math. Mechanics}, 14(1):117--140, 1965.

\bibitem{kodaira-embedding}
Kunihiko Kodaira.
\newblock On {K}\"ahler varieties of restricted type ({A}n intrinsic
  characterization of algebraic varieties).
\newblock {\em Ann. Math.}, 60(1):28--48, 1954.

\bibitem{magnusson2012natural}
Gunnar~\TH\'or Magn{\'u}sson.
\newblock A natural hermitian metric associated with local universal families
  of compact {K}{\"a}hler manifolds with zero first {C}hern class.
\newblock {\em Comptes Rendus Math.}, 350(1-2):63--66, 2012.

\bibitem{stoica2011cartan}
Cristi Stoica.
\newblock Cartan’s structural equations for degenerate metric.
\newblock {\em arXiv preprint math.DG/1111.0646}, 7, 2011.

\bibitem{stoica2014singular}
Ovidiu~Cristinel Stoica.
\newblock On singular semi-riemannian manifolds.
\newblock {\em International Journal of Geometric Methods in Modern Physics},
  11(05):1450041, 2014.

\bibitem{zheng2000complex}
Fangyang Zheng.
\newblock {\em Complex differential geometry}.
\newblock Number~18. Amer. Math. Soc., 2000.

\end{thebibliography}

\end{document}